\documentclass{amsart}
\usepackage[utf8]{inputenc}
\usepackage{amssymb,amsmath,amsthm,latexsym}
\usepackage{graphicx}
\usepackage[dvipsnames]{xcolor}
\usepackage{tikz}
\usepackage{amscd}
\usepackage{setspace}
\usepackage{comment}
\usepackage{hyperref}
\usepackage{cleveref}
\usepackage{algorithm}
\usepackage[noend]{algpseudocode}
\usepackage{mathdots}
\usepackage[a4paper,inner=3cm,outer=3cm,top=2.5cm,bottom=2.5cm]{geometry}

\newtheorem{theorem}{Theorem}[section]
\newtheorem*{theorem*}{Theorem}
\newtheorem{proposition}[theorem]{Proposition}
\newtheorem{lemma}[theorem]{Lemma}
\newtheorem{corollary}[theorem]{Corollary}
\newtheorem{claim}[theorem]{Claim}

\theoremstyle{definition}
\newtheorem{definition}[theorem]{Definition}

\theoremstyle{remark}
\newtheorem{remark}[theorem]{Remark}
\newtheorem{example}[theorem]{Example}

\newcommand{\I}{\mathcal{I}}

\begin{document}

\title{Matchings in Matroids over Abelian Groups, III}

\author[M. Aliabadi]{Mohsen Aliabadi}
\address{Department of Mathematics, Clayton State University, Morrow, Georgia, USA}
\email{maliabadi@clayton.edu}

\author[E. Krop]{Elliot Krop}
\address{Department of Mathematics, Clayton State University, Morrow, Georgia, USA}
\email{ElliotKrop@clayton.edu}

\dedicatory{Dedicated to Melvyn Nathanson and Carl Pomerance on their 80th birthdays}

\begin{abstract}
This paper develops matroidal analogues of classical results on matchings in abelian groups. By embedding matroid ground sets in an abelian group, we introduce base matchings between matroid bases, recover the group-theoretic setting in the uniform matroid case, and derive structural and combinatorial criteria for their existence. Our main focus is on paving matroids. We prove self-matchability for paving matroids, extend asymmetric matchability results using the hyperplane-nullity parameter, and show that stressed hyperplanes provide a natural route to matchability through relaxation.
\end{abstract}

\maketitle

\section{Introduction}

\subsection{Overview}

In an abelian group \(G\), a \emph{matching} is a bijection \(f\colon A\to B\) between finite subsets \(A,B\subseteq G\) such that $a+f(a)\notin A$ for all $a\in A$. We say that \(G\) has the \emph{matching property} if, for every pair of finite subsets \(A,B\subseteq G\) with \(|A|=|B|\) and \(0\notin B\), there exists such a matching. A geometric model for a class of bipartite graphs was introduced in \cite{Fan}. In that work, a special kind of perfect matching, called an acyclic matching, was defined and shown to exist for a subclass of these graphs by geometric methods. The existence of acyclic matchings implies the nonvanishing of the determinant of certain weighted biadjacency matrices.
This matching framework was later applied to a question posed by E.~K. Wakeford in 1920 \cite{Wakeford}, concerning which sets of monomials can be eliminated from a generic homogeneous polynomial by linear changes of variables. In a special case related to canonical forms of symmetric tensors, Fan and Losonczy reduced Wakeford's problem to the existence of acyclic matchings in \(\mathbb{Z}^n\). This property was established for \(\mathbb{Z}^n\) by Alon, Fan, Kleitman, and Losonczy~\cite{Alon}, further developed in~\cite{Aliabadi0, Taylor 1}, and later completely characterized for all abelian groups in~\cite{Taylor}.

Matchings were extended to the abelian group setting in \cite{Losonczy} and to non-abelian groups in \cite{Eliahou1}, with further developments, particularly on enumerative aspects, presented in \cite{Hamidoune}. A linear formulation was later introduced in \cite{Eliahou2}, and a matroidal analogue was proposed in \cite{Zerbib0}. Additional progress on matchable matroids appears in \cite{Rita}. In particular, \cite{Zerbib0,Rita} investigated matchability for the following four classes of matroids:
\begin{itemize}
    \item sparse paving matroids,
    \item transversal matroids,
    \item panhandle matroids,
    \item Schubert matroids.
\end{itemize}
In this paper, we continue this line of research by studying the matchability of paving matroids. Following the terminology introduced in \cite{Losonczy}, we present the definition of matchings in abelian groups.

\noindent\textit{Matchings in abelian groups.}  Throughout, we use the notation \([n] := \{1, \dots, n\}\) for any positive integer \(n\). 
We also write \((G,+)\) for an abelian group with neutral element \(0\).
 Let $p(G)$ represent the smallest cardinality of a nontrivial (i.e., non-zero) finite subgroup of $G$. If $G$ has no nontrivial finite subgroup, we make the convention $p(G)=\infty$. Let $A$ and $B$ be finite subsets of $G$ with equal cardinality, such that $0 \notin B$. A \emph{matching} from $A$ to $B$ is a bijection $f: A \to B$ satisfying the condition that $a + f(a) \notin A$ for every $a \in A$. The conditions $|A| = |B|$ and $0 \notin B$ are necessary for the existence of such a bijection. A matching is called \emph{symmetric} if $A = B$, and \emph{asymmetric} otherwise. If there exists a matching from $A$ to $B$, then $A$ is said to be \emph{matched} or \emph{matchable} to $B$. An abelian group $G$ is said to satisfy the \emph{matching property} if for every pair of finite subsets $A, B \subseteq G$ with $|A| = |B|$ and $0 \notin B$, the set $A$ is matched to $B$.

Characterizations of abelian groups satisfying the matching property, as well as necessary and sufficient conditions for the existence of symmetric matchings, were obtained by Losonczy.

\begin{theorem}[\cite{Losonczy}]\label{symmetric matching}
Let $G$ be an abelian group, and let $A$ be a nonempty finite subset of $G$. Then there exists a matching from $A$ to itself if and only if $0 \notin A$.
\end{theorem}

\begin{theorem}[\cite{Losonczy}]\label{matching property}
An abelian group $G$ satisfies the matching property if and only if $G$ is torsion-free or cyclic of prime order.
\end{theorem}

Sufficient conditions for the existence of matchings between small enough finite sets in abelian groups were provided in \cite{Aliabadi2}.

\begin{theorem}[\cite{Aliabadi2}]\label{matchable sets, p(G)}
Let $G$ be an abelian group, and let $A$ and $B$ be finite subsets of $G$ with $|A| = |B| < p(G)$ and $0 \notin B$. Then there exists a matching from $A$ to $B$.
\end{theorem}

\noindent\textit{A connection to matchings in graphs.} While there are meaningful connections between the notion of matchings explored in this paper and the classical matching theory in graphs introduced by Philip Hall in 1935, the two concepts are fundamentally distinct. In particular, in order to address a linear algebra problem posed by Wakeford~\cite{Wakeford}, Fan~\cite{Fan} constructed a bipartite graph
\[
\mathcal{G}=(V(\mathcal{G}),E(\mathcal{G}))
\]
from two sets \(A,B\subseteq \mathbb{Z}^n\). To each element \(a\in A\) one associates a symbol \(x_a\), and to each element \(b\in B\) a symbol \(y_b\). Let
\[
X=\{x_a \mid a\in A\}
\qquad\text{and}\qquad
Y=\{y_b \mid b\in B\}.
\]
The graph \(\mathcal{G}\) is then defined to have vertex set
\[
V(\mathcal{G})=X\cup Y,
\]
where an edge joins \(x_a\in X\) to \(y_b\in Y\) if and only if
\[
a+b\notin A.
\]

In graph theory, a \textit{matching} is a set of edges with no shared endpoints, and a \textit{perfect matching} is one that covers all vertices. In the bipartite graph with bipartition sets $V(\mathcal{G})=X\cup Y$, indexed by $A$ and $B$, perfect matchings correspond to bijections \( f: A \to B \) such that \( a + f(a) \notin A \) for every \( a \in A \). Wakeford's problem \cite{Wakeford} thus becomes equivalent to determining the existence of such a bijection, referred to as an acyclic matching from \( A \) to \( B \).

It is worth noting that, in the classical graph-theoretic setting, a perfect matching merely requires that every vertex is incident to exactly one edge, without reference to a bijection \( f: A \to B \) or to the additional condition \( a + f(a) \notin A \). By contrast, the notion of matchings in abelian groups incorporates this extra structure in order to capture the linear-algebraic framework underlying Wakeford's problem.

\subsection{Matching in Matroids}\label{Matching in matroids}
We begin with the definition of a matroid. The reader is referred to \cite{Oxley} for all undefined terms. A {\it matroid} $M$ is a pair $(E, \mathcal{I})$ where $E=E(M)$ is a finite  {\em ground set}  and $\mathcal{I}$ is a family of subsets of $E$, called {\em independent sets},  satisfying the following  conditions:
\begin{itemize}
\item $\emptyset \in \I$.
\item If $X\in \I$ and $Y\subseteq X$ then $Y\in \I$.
\item The {\em augmentation property}: If $X, Y \in \I$  and $|X|>|Y|$ then there exists $x\in X\setminus Y$ so that $Y\cup \{x\} \in \I$. 
\end{itemize}

Let $M=(E, \mathcal{I})$ be a matroid. 
The {{\em rank}} of a subset $X\subseteq E$ is given by
$$r_M(X)=
r(X)=max\{|X\cap I|: I\in \mathcal{I}\}.
$$

The rank of a matroid $M$, denoted by $r(M)$, is defined to be $r_M(E(M))$.
A set $X\subseteq E$ is called {{\em dependent}} if it is not independent.
A maximal independent set is called a {{\em basis}}. It follows from the augmentation property that every two bases have the same size. A minimal dependent set is called a {\it{circuit}}. An element that forms a single-element circuit of $M$ is called a {\it{loop}}. A matroid is called {\emph{loopless}} if it does not have a loop.

\bigskip
Given a matroid $M=(E,\I)$, the {{\em dual matroid}} $M^*=(E, \mathcal{I}^*)$ is defined so that the
bases in $M^*$ are exactly the complements of the bases in $M$. A matroid $M=(E,\I)$ of rank $n$ is said to be  a {{\em paving matroid}} if every $(n-1)$-subset of $E$ is independent. If $M$ and $M^*$ are both paving matroids, then $M$ is called a {\em sparse paving matroid}.\\
A {{\em flat}} in  $M$ is a set $F \subseteq E$ with the property that adjoining any new element to $F$ strictly increases its rank. A flat of rank $r(M)-1$ is called a {{\em hyperplane}}. A set $H\subseteq E(M)$ is called a {\it {circuit-hyperplane}} of $M$ if it is at the same time a circuit and a hyperplane.
For any $X\subseteq E$ the difference $|X|-r(X)$ is called the {\it{nullity}} of $X$ and is denoted by ${\it{null}}(X)$. Let us denote the set of all hyperplanes of $M$ by $\mathcal{H}_M$. Then the {\it{hyperplane nullity}} of $M$ is denoted by ${\it{null}}(\mathcal{H}_M)$ and defined as 
$$
{\it{null}}(\mathcal{H}_M)=\max\{{\it{null}}(H): H\in \mathcal{H}_M\}.
$$
For integers $m$ and $n$ with $0 \leq n \leq m$, the \emph{uniform matroid} $U_{n,m}$ is defined on a ground set $E$ of size $m$ such that a subset $A \subseteq E$ is independent if and only if $|A| \leq n$. Equivalently, the bases of $U_{n,m}$ are exactly the subsets of $E$ of size $n$. Note that uniform matroids are a very special class of sparse paving matroids. A uniform matroid $U_{n,m}$ is called a {\it{free matroid}} if $m=n$. 

\bigskip
If $X\subseteq E$, the {\em closure} of $X$ is denoted by ${\it{cl}}(X)$ and defined as follows:
$$
{\it{cl}}(X)=\{x\in E: r(X)=r(X\cup \{x\})\}.
$$
The closure operator satisfies the following properties:
\begin{itemize}
    \item For all subsets $X$ of $E$, $X\subseteq {\it{cl}}(X)$.
    \item If $X$ and $Y$ are subsets of $E$ with $X\subseteq Y$, then ${\it{cl}}(X)\subseteq {\it{cl}}(Y)$.
    \item If $r(X)=r(M)-1$, then ${\it{cl}}(X)$ is a hyperplane of $M$.
\end{itemize}
\bigskip
 Following \cite{Zerbib0} we introduce the definition of matchings in matroids over an abelian group. 
 We may assume that all matroids are loopless (see Remark 5.1 in \cite{Zerbib0}.)
We say that {\emph{$M=(E,\mathcal{I})$ is a matroid over $G$}} if $E$ is a subset of $G$. 

\begin{definition}\label{def, matching matroid}  Let  $(G,+)$ be an abelian group.
\begin{enumerate}
\item  Let $M$ and $N$ be two matroids over $G$ with $r(M)=r(N)=n>0$. Let $\mathcal{M}=\{a_1,\ldots,a_n\}$ and $\mathcal{N}=\{b_1,\ldots,b_n\}$ be ordered bases of $M$ and $N$, respectively. We say $\mathcal{M}$ is {{\em matched}} to $\mathcal{N}$ if $a_i+b_i\notin E(M)$, for all $i\in [n]$. 
    
    \item We say that $M$ is {{\em matched}}  to $N$ if for every basis $\mathcal{M}$ of $M$ there exists a basis $\mathcal{N}$ of $N$ such that $\mathcal{M}$ is matched to $\mathcal{N}$.

 \end{enumerate}
\end{definition}

    Note that adopting the notation from Definition~\ref{def, matching matroid}, suppose \( \mathcal{M} \) is matched to \( \mathcal{N} \) as matroid bases. Then for each \( 1 \leq i \leq n \), we have \( a_i + b_i \notin E(M) \), and in particular \( a_i + b_i \notin \mathcal{M} \). This implies that the map \( a_i \mapsto b_i \) defines a matching in the group-theoretic sense between the subsets \( \mathcal{M} \) and \( \mathcal{N} \) of \( G \). In this way, the definition of matchings in matroids is compatible with the classical notion of matchings in abelian groups.

\bigskip

\noindent\textit{Compatibility of matchings in matroids with the vector space setting.}
Definition~\ref{def, matching matroid} is consistent with the notion of matchability in vector spaces over field extensions, as introduced by Eliahou and Lecouvey~\cite{Eliahou2}. Let \( K \subseteq F \) be a field extension, and let \( A, B \subseteq F \) be \( n \)-dimensional \( K \)-vector subspaces. Suppose \( \mathcal{A} = \{a_1, \dots, a_n\} \) and \( \mathcal{B} = \{b_1, \dots, b_n\} \) are ordered bases of \( A \) and \( B \), respectively. Then \( \mathcal{A} \) is said to be \emph{matched} to \( \mathcal{B} \) if
\begin{align}\label{match vector}
    a_i^{-1} A \cap B \subseteq \langle b_1, \dots, \hat{b}_i, \dots, b_n \rangle \quad \text{for every } i\in [n],
\end{align}
where \( \langle b_1, \dots, \hat{b}_i, \dots, b_n \rangle \) denotes the \( K \)-vector subspace spanned by \( \mathcal{B} \setminus \{b_i\} \). The space \( A \) is said to be matched to \( B \) if every basis \( \mathcal{A} \) of \( A \) is matched to some basis \( \mathcal{B} \) of \( B \) in this way.

Now let \( G = F \setminus \{0\} \) be the multiplicative group of \( F \), and define matroids over \( G \) by setting \( M = (A \setminus \{0\}, \mathcal{I}) \) and \( N = (B \setminus \{0\}, \mathcal{I}') \), where \( \mathcal{I} \) and \( \mathcal{I}' \) are the families of linearly independent subsets of \( A \) and \( B \), respectively. Then
\[
r(M) = \dim_K A = n = \dim_K B = r(N).
\]
Suppose the basis \( \mathcal{A} \) is matched to \( \mathcal{B} \) in the vector space setting. It follows from condition~\eqref{match vector} that \( a_i b_i \notin A \) for all \( i \), and hence \( A \) is also matched to \( B \) in the matroid-theoretic sense.

\bigskip

In this paper we broaden the scope of matroid matchability in three directions. 
First, we extend the symmetric paving case by proving that any paving matroid is matched to itself under suitable size conditions (Theorem~\ref{paving sym}). 
Second, we develop an asymmetric extension in which matchability is controlled by the hyperplane-nullity parameter~$t$, generalizing the sparse paving setting (Theorem~\ref{Asy pav}). 
Finally, we show that stressed hyperplanes provide a natural route to matchability in relaxations of paving matroids (Theorem~\ref{Asy paving}). 

The point of passing from sparse paving to paving matroids is not merely to enlarge the class of matroids for which matching results are known. In sparse paving matroids (with rank $n$), the only dependent \(n\)-sets are circuit-hyperplanes, so the obstruction is essentially binary. In other words, the failure of an \(n\)-subset to be a basis is controlled by a single type of obstruction, namely circuit-hyperplanes. By contrast, in paving matroids, hyperplanes may have larger nullity, and our results show that matchability is controlled by this richer geometric parameter. In this way, the paper shows that the key mechanism governing matroid matchings is the structure of hyperplanes, rather than near-uniformity alone.

\medskip

\noindent\textit{Organization of the paper.} 
Section~\ref{prelim} reviews preliminaries on paving and sparse paving matroids, including their characterization via partitions and hyperplane nullity, as well as additive tools such as Kneser’s theorem and results on critical pairs. 
Section~\ref{main result} presents the main theorems: Subsection~\ref{sym} treats the symmetric case, Subsection~\ref{asym} develops the asymmetric theory incorporating hyperplane nullity, and Subsection~\ref{stressed} introduces stressed hyperplanes and relaxation, connecting paving matroids to uniform matroids and demonstrating how these ideas ensure broader matchability. 

\section{Preliminaries}\label{prelim}
\subsection{Notions and Results on Paving Matroids}\label{matroid general}
A {\it{$d$-partition}} of a finite set $E$ is a collection $\mathcal{S}$ of subsets of $E$, all of which are of size at least $d$, so that the intersection of any two distinct sets in $\mathcal{S}$ is of size at most $d-1$. The set $\mathcal{S} = \{E\}$ is a $d$-partition for any $d$, which we call a {\it{trivial
$d$-partition.}}
The connection between paving matroids and d-partitions is given by the following result, whose proof can be found in \cite{Oxley}.

\begin{theorem} \cite{Oxley}\label{(n-1) partitions}
If $M$ is a paving matroid of rank $n\geq 2$, then its hyperplanes form
a non-trivial $(n-1)$-partition of $E(M)$. That is, the intersection of any two distinct hyperplanes in $M$ is of size at most $n-2$.
\end{theorem}
The following result provides an alternative characterization of sparse paving matroids.
\begin{theorem}\cite{Oxley}\label{sparse paving alternative}
A matroid $M$ of rank $n$ is a sparse paving matroid if and only if every $n$-subset of $E(M)$ is either a basis or a circuit-hyperplane.
\end{theorem}

Returning to the nullity of hyperplanes, note that if $M$ is a uniform matroid, then clearly $\it{null}(\mathcal{H}_M)=0$. In other words, hyperplane nullity measures how far a matroid deviates from being uniform. In the more general case where $\it{null}(\mathcal{H}_M)\leq 1$, we obtain the following characterization, which follows directly from Theorem \ref{sparse paving alternative} and highlights the role of hyperplane nullity in identifying sparse paving matroids.

\begin{lemma}\label{sparse paving nullity}
A paving matroid $M$ is sparse paving if and only if ${\it{null}(\mathcal{H}_M)}\leq 1$.
\end{lemma}

\begin{proof}
Suppose first that $M$ is a sparse paving matroid. Let $H$ be a hyperplane of $M$. Clearly $|H|\geq n-1$.  
If $|H|=n-1$, then $H$ is independent and hence ${\it null}(H)=0$.  
If $|H|>n-1$, then $H$ contains an $n$-subset $H'$. By Theorem \ref{sparse paving alternative}, $H'$ is itself a hyperplane, and thus $H=H'$, implying $|H|=n$. In this case, ${\it null}(H)=1$.  
Therefore, in either case we have ${\it null}(\mathcal{H}_M)\leq 1$.

Conversely, assume that ${\it null}(\mathcal{H}_M)\leq 1$. Let $A$ be an $n$-subset of $E(M)$ that is not a basis. We claim that $A$ is a circuit-hyperplane. Since $M$ is paving, then $A$ is a circuit with $r(A)=n-1$. Consider the hyperplane $H=\operatorname{cl}(A)$. Then ${\it null}(H)\leq 1$. We have
\[
1=|A|-r(A) \leq |H|-r(H) = {\it null}(H) \leq {\it null}(\mathcal{H}_M) \leq 1.
\]
Hence, all inequalities are equalities, which forces $A=H$. Thus $A$ is a hyperplane. Since $A$ is a circuit-hyperplane, by Theorem \ref{sparse paving alternative}, $M$ is sparse paving.
\end{proof}

\begin{corollary}\label{small nullity}
Let $M$ be a paving matroid. Then
\begin{enumerate}
    \item $M$ is uniform if and only if ${\it{null}}(\mathcal{H}_M)=0$.
    \item $M$ is non-uniform  and sparse paving if and only if  ${\it{null}}(\mathcal{H}_M)=1$.
\end{enumerate}
\end{corollary}
\begin{proof}
It is immediately seen by invoking Lemma \ref{sparse paving nullity}.
\end{proof}

\subsection{Abelian Additive Theory} 
The main additive number theory tools used in the proofs of our main results are Theorem \ref{Kneser} and Lemma \ref{Kemperman's consequences}.

Let $G$ be an abelian group and $A,B\subseteq G$. The sumset $A+B$ is
\[
A+B=\{\,a+b : a\in A,\; b\in B\,\}.
\]

We start with a result of Kneser\cite[Theorem 4.3]{Nathanson} providing a lower bound on the size of $A+B$.
\begin{theorem}\label{Kneser}
    For an abelian group $G$ and finite nonempty subsets $A,B \subseteq G$, 
the sumset $A+B$ satisfies
\[
|A+B| \;\geq\; |A| + |B| - |H|,
\]
where 
\[
H = \operatorname{Stab}(A+B) = \{\, g \in G : (A+B)+g = A+B \,\}
\]
is the stabilizer of $A+B$.

\end{theorem}
Let $G$ be an abelian group, $x, a\in G$ and let $k$ be a positive integer.  A {\em progression} of length $k$ with {\em difference} $x$ and {\em initial term} $a$ is a subset of $G$ of the form
$
\{a,a+x,a+2x,...,a+(k-1)x\}
$.  We say $A$ is a {\em{semi-progression}} if $A\setminus\{a\}$ is a progression, for some $a\in A$.
In the following lemma due to Kemperman, a certain family of critical pairs is characterized. Note that a pair $(A,B)$ of finite subsets of a group $G$ is called {\it{critical}} if $|G|>|A+B|=|A|+|B|-1$.
\begin{lemma} \cite{Kemperman2}\label{Kemperman's critical theorem} Let $(A,B)$ be a critical pair of a finite abelian group $G$ with $|A|>1$, $|B|>1$ and $|A|+|B|-1\leq p(G)-2$. Then $A$ and $B$ are progressions with the same difference.
\end{lemma}
The following classical theorem was proved by Kemperman \cite{Kemperman}.
\begin{theorem} \cite{Kemperman}\label{Kemperman's theorem}
Let $A$ and $B$ be nonempty finite subsets of a group $G$. Assume there exists an element $c\in A+B$ appearing exactly once as a sum
$c = a+b$ with $a\in A, b\in B$.   Then 
\begin{align*}
|A+B|\geq|A|+|B|-1.
\end{align*}
\end{theorem}
The following follows from Theorem \ref{Kemperman's theorem}. 
\begin{lemma}\label{Kemperman's consequences}
Let $A$ and $B$ be nonempty finite subsets of a group $G$ such that $A$,  $B$ and $A+B$ are all contained in a subset $X$ of $G\setminus\{0\}$. Then 
\begin{align*}
|X|\geq|A|+|B|+1.
\end{align*}
\end{lemma}
\begin{proof}
Define $A_0=A\cup\{0\}$ and $B_0=B\cup\{0\}$. Then $0\in A_0+B_0$ and appears exactly once as a sum in $A_0+B_0$. That is, if $0=a+b$ with $a\in A$ and $b\in B$, since $0\notin A+B$, then $a=0$ or $b=0$, and hence $a=b=0$. Applying Theorem \ref{Kemperman's theorem} to $A_0$ and $B_0$ we have 
$$
|A_0+B_0|\geq |A_0|+|B_0|-1.
$$
Set $X=(A+B)\cup A\cup B$. Since $|A_0|=|A|+1$, $|B_0|=|B|+1$ and $A_0+B_0=(A+B)\cup A\cup B$, we have:
$$
|X|=|A_0+B_0|\geq |A|+|B|+1,
$$
as claimed.
\end{proof}

\section{Main Results}\label{main result}
 Throughout, all matroids are assumed to be loopless. In this section, we investigate matchability within paving matroids. Although paving matroids possess a seemingly restrictive structure, Crapo and Rota \cite{Crapo} conjectured, and it is now widely believed, that asymptotically almost all matroids are paving. While this conjecture remains unresolved, it is supported by strong asymptotic estimates \cite{Pendavingh2015, Pendavingh2017}. What is certain, however, is that paving matroids constitute a large and significant class of matroids.

 Note that matchability in the group setting can be viewed as a special case of matchability in the matroid setting. Let \( A \) and \( B \) be two finite nonempty subsets of an abelian group \( G \), each of cardinality \( n \), and assume \( 0 \notin B \). Define a uniform matroid \( M \cong U_{n,n} \) on the ground set \( A \), and similarly define \( N \cong U_{n,n} \) on the ground set \( B \). Then \( A \) is matched to \( B \) in the group-theoretic sense if and only if \( M \) is matched to \( N \) in the matroid-theoretic sense. Note that uniform matroids form a very narrow subclass of all matroids, and \( U_{n,n} \) represents a particularly special case.
    
The importance of paving matroids is not only that they extend sparse paving matroids, but that they make the structural source of matchability more visible. Since the classical group-theoretic notion of matching is recovered in the uniform case, paving matroid matchability extends the original problem of matchings within groups. Passing to all paving matroids, therefore, helps separate features of matchability that depend on near-uniformity from those that persist in a richer dependence structure.
Our results show that, in paving matroids, matchability is governed by geometric features such as hyperplane intersections, hyperplane nullity, and stressed hyperplanes. This both explains why the theory extends beyond the sparse paving case and places the classical group case within a broader structural framework. Since paving matroids can also be related back to uniform matroids through relaxation (Theorem \ref{Asy paving}), the group-theoretic setting appears not as an isolated phenomenon, but as the first instance of a more general theory of matchings.

\subsection{Matching Paving Matroids over Abelian Groups (Symmetric Case)}\label{sym}

The following result was proved in \cite{Zerbib0} for matchings in sparse paving matroids under the symmetric case, providing a matroidal version of Theorem \ref{symmetric matching}.

\begin{theorem}\label{sparse sym main}
Let \(M\) be a sparse paving matroid over \(G\).  Then \(M\) is matched to itself if and only if \(0 \notin E(M)\).
\end{theorem}

We now extend this to the paving matroid analogue of Theorem~\ref{sparse sym main}.

\begin{theorem}\label{paving sym}
Let \(M\) be a paving matroid over \(G\).  Then \(M\) is matched to itself if and only if \(0 \notin E(M)\).
\end{theorem}

Note that, for any pair of matroids $M,N$, the conditions $0\notin E(N)$ and $r(M)=r(N)$ 
are necessary for $M$ to be matched to $N$ (see~\cite[Proposition~1.7]{Zerbib0}). 
Thus, the ``only if'' implication in Theorem~\ref{paving sym} is immediate. 
We now establish the reverse implication.

\begin{proof}
 We assume that $M$ is not a free matroid (the case when $M$ is free is addressed in Remark~\ref{free matroid}). 
Since $0 \notin E(M)$, it follows from Theorem~\ref{symmetric matching} that $E(M)$ is matched to itself in the group sense. 
That is, there exists a bijection $f : E(M) \to E(M)$ with $a + f(a) \notin E(M)$ for all $a \in E(M)$. 

Let $\mathcal{M} = \{a_1, \dots, a_n\}$ be a basis of $M$. 
Set $f(a_i) := b_i$ and define $\mathcal{N} = \{b_1, \dots, b_n\}$. 
If $\mathcal{N}$ is a basis of $M$, then $\mathcal{M}$ is matched to $\mathcal{N}$ in the matroid sense and we are done. 
Otherwise, suppose that $\mathcal{N}$ is not a basis of $M$.

\begin{claim}\label{extr}
There exist $x \in E(M) \setminus \mathcal{N}$ and $i \in [n]$ such that $a_i + x \notin E(M)$.
\end{claim}
\begin{proof}[Proof of Claim]
Suppose not. Then 
\[
\mathcal{M} + (E(M) \setminus \mathcal{N}) \subseteq E(M).
\]
Since the three sets $\mathcal{M}$, $E(M)\setminus \mathcal{N}$, and $\mathcal{M} + (E(M)\setminus \mathcal{N})$ are contained in $E(M)$ and $0 \notin E(M)$, 
Lemma~\ref{Kemperman's consequences} yields
\[
|E(M)|
   \;\geq\; |\mathcal{M}| + |E(M)\setminus \mathcal{N}| + 1
   \;=\; |E(M)| + 1,
\]
which is a contradiction.
This completes the proof of the claim.
\end{proof}

    \bigskip
    
    Let $x$ and $i$ be as given by Claim~\ref{extr}. Without loss of generality, assume $i = n$ and define
\[
\mathcal{P} = \{b_1, \dots, b_{n-1}, x\}.
\]

\begin{claim}
$\mathcal{P}$ is a basis of $M$.
\end{claim}

\begin{proof}[Proof of the Claim]
Suppose not. Since $M$ is a paving matroid, we must then have $r(\mathcal{P}) = n-1$. 
Consider the hyperplanes $H_1 = \mathrm{cl}(\mathcal{N})$ and $H_2 = \mathrm{cl}(\mathcal{P})$. 
By Theorem~\ref{(n-1) partitions}, it follows that
\[
|H_1 \cap H_2| \leq n-2.
\]
However, it is clear that $|H_1 \cap H_2| \geq n-1$, since $b_i \in H_1 \cap H_2$ for all $i \in [n-1]$. 
This contradiction proves that $\mathcal{P}$ is indeed a basis.
\end{proof}

\medskip

Finally, observe that the basis $\mathcal{M}$ is matched to the basis $\mathcal{P}$. 
Hence $M$ is matched to itself, completing the proof.

\end{proof}
\begin{remark}\label{free matroid}
In the proof of Theorem~\ref{paving sym}, if $M$ is a free matroid, then it has a unique basis, namely $E(M)$. 
Since $0 \notin E(M)$, Theorem~\ref{symmetric matching} guarantees that $E(M)$ is matched to itself in the group sense. 
Consequently, the basis $E(M)$ is also matched to itself in the matroid sense. 
Hence $M$ is matched to itself.
\end{remark}
\begin{remark}
The significance of Theorem \ref{paving sym} is that it shows the classical self-matching criterion from abelian groups (Theorem \ref{symmetric matching}) is not confined to sets or to uniform matroids. Even in a paving matroid, where substantial dependence may occur, the only obstruction to symmetric self-matchability remains the trivial additive obstruction \(0 \in E(M)\). Thus, the theorem identifies the paving condition as a structural setting in which matroidal dependence does not create new symmetric matching obstructions. In particular, it shows that the source of self-matchability lies not merely in near-uniformity, but in the hyperplane geometry of paving matroids.
\end{remark}

\begin{example}
We construct a paving matroid over $\mathbb{Z}$ whose ground set does not contain $0$, and then show that it is matched to itself in view of Theorem~\ref{paving sym}.  
To illustrate a case not covered by Theorem~\ref{sparse sym main}, we define a paving matroid that is not sparse paving.

Consider the matroid $M$ of rank $3$ with ground set $E(M)=[5]$ and bases
\[
\mathcal{B}(M)=\bigl\{\{i,j,5\}\mid i,j\in [4],\ i\neq j\bigr\}.
\]
The $3$-circuits are
\[
\mathcal{C}_3(M)\;=\;\big\{\{1,2,3\},\ \{1,2,4\},\ \{1,3,4\},\ \{2,3,4\}\big\}.
\]

Thus $M$ is a paving matroid of rank $3$. However, it is not sparse paving, since the subset $\{1,2,3\}$ is a $3$-subset that is neither a basis nor a hyperplane.  

We use the construction of Theorem \ref{paving sym} to show that $M$ is matched to itself. First, consider the matching
\[
f:[5]\to [5], \qquad 1\mapsto 5,\; 2\mapsto 4,\; 3\mapsto 3,\; 4\mapsto 2,\; 5\mapsto 1,
\]
in the group setting.
\begin{itemize}
    \item Consider the basis $\mathcal{M}=\{1,2,5\}$. Then
    \[
    \mathcal{N}=f(\mathcal{M})=\{5,4,1\},
    \]
    which is a basis of $M$, and $\mathcal{M}$ is matched to $\mathcal{N}$.
    \item Consider the basis $\mathcal{M}=\{1,3,5\}$. Then
    \[
    \mathcal{N}=f(\mathcal{M})=\{5,3,1\},
    \]
    which is a basis of $M$, and $\mathcal{M}$ is matched to $\mathcal{N}$.
    \item Consider the basis $\mathcal{M}=\{1,4,5\}$. Then
    \[
    \mathcal{N}=f(\mathcal{M})=\{5,2,1\},
    \]
    which is a basis of $M$, and $\mathcal{M}$ is matched to $\mathcal{N}$.
    
    \item Consider the basis $\mathcal{M}=\{2,3,5\}$. Then
\[
\mathcal{N}=f(\mathcal{M})=\{4,3,1\},
\]
which is not a basis of $M$. Choose $a_3=5$ and $x=5$ as in Claim \ref{extr}, and form the basis
\[
\mathcal{P}=\{4,3,5\}
\]
of $M$. Then $\mathcal{M}$ is matched to $\mathcal{P}$.

\item Consider the basis $\mathcal{M}=\{2,4,5\}$. Then
\[
\mathcal{N}=f(\mathcal{M})=\{4,2,1\},
\]
which is not a basis of $M$. Choose $a_1=2$ and $x=5$ as in Claim \ref{extr}, and form the basis
\[
\mathcal{P}=\{5,2,1\}
\]
of $M$. Then $\mathcal{M}$ is matched to $\mathcal{P}$.
\item Consider the basis $\mathcal{M}=\{3,4,5\}$. Then
\[
\mathcal{N}=f(\mathcal{M})=\{3,2,1\},
\]
which is not a basis of $M$. Choose $a_2=4$ and $x=5$ as in Claim \ref{extr}, and form the basis
\[
\mathcal{P}=\{3,5,1\}
\]
of $M$. Then $\mathcal{M}$ is matched to $\mathcal{P}$.

\end{itemize}
Since every basis $\mathcal{M}$ of $M$ is matched to a basis of $M$, it follows that $M$ is matched to itself.
\end{example}
Let $E$ be a nonempty set. A \emph{matroid basis system} is a nonempty family $\mathcal{B}\subseteq 2^{E}$ of subsets, called bases, that satisfies the \emph{exchange axiom}: for all distinct $B,B'\in\mathcal{B}$ and every $x\in B\setminus B'$, there exists $y\in B'\setminus B$ such that $(B\setminus\{x\})\cup\{y\}\in\mathcal{B}$. It follows that $\mathcal{B}$ is the set of bases of a matroid (see Kung’s survey \cite{Kung 1}).
\begin{example}
In this example we give a non-paving matroid over an abelian group for which the construction used in the proof of Theorem~\ref{paving sym} breaks down.

Let $G=\mathbb{Z}$ and let
\[
E=\{1,3,7,9\}.
\]
Define a matroid $M$ of rank $3$ on $E$ by declaring its bases to be exactly
\[
\mathcal{B}(M)=\bigl\{\{1,7,9\},\{3,7,9\}\bigr\}.
\]
It is immediate that this is a matroid: the collection $\mathcal{B}(M)$ is nonempty, and the basis exchange axiom is trivially satisfied since the only two bases differ by exchanging $1$ and $3$.

Since the $2$-subset $\{1,3\}$ is contained in no basis, it is dependent. Hence $M$ is not paving.

Now define a matching $f:E\to E$ in the group setting by
\[
f(1)=9,\qquad f(3)=1,\qquad f(7)=3,\qquad f(9)=7.
\]

Consider the basis
\[
\mathcal{M}=\{3,7,9\}.
\]
Following the construction in the proof of Theorem~\ref{paving sym}, set
\[
\mathcal{N}=f(\mathcal{M})=\{1,3,7\}.
\]
Then $\mathcal{N}$ is not a basis of $M$, since $\mathcal{N}\notin \mathcal{B}(M)$.

Now take $a_3=9$ and $x=9\in E\setminus\mathcal{N} $. Since
\[
9+9=18\notin E,
\]
the analogue of Claim~\ref{extr} holds. Replacing the corresponding element $7$ of $\mathcal{N}$ by $x=9$, we obtain
\[
\mathcal{P}=\{1,3,9\}.
\]
Again, $\mathcal{P}$ is not a basis of $M$, since $\mathcal{P}\notin \mathcal{B}(M)$.

Therefore, the replacement step used in the proof of Theorem~\ref{paving sym} may fail for non-paving matroids. This shows that the paving assumption is needed in order to conclude that the modified set is again a basis.
\end{example}
\subsection{Matching Paving Matroids over Abelian Groups (Asymmetric Case)}\label{asym}

The following result about the matchability of sparse paving matroids in the asymmetric situation was proved in \cite{Zerbib0}. 

\begin{theorem}\label{Asy sparse paving}
Let $M$ be a matroid over $G$ and let $N$ be a sparse paving matroid over $G$, both of the same rank $n$, with $0 \notin E(N)$. 
Assume further that one of the following conditions holds:
\begin{enumerate}
    \item $|E(M)| < \min\{|E(N)| - 1,\, p(G)\}$, or
    \item $E(M)$ is not a progression, $G$ is finite, $|E(M)| = |E(N)| - 1$, and $|E(N)| < p(G)$, or
    \item $E(M)$ is neither a progression nor a semi-progression, $G$ is finite, $|E(M)| = |E(N)| < p(G)$, or
    \item $|E(M)| < |E(N)| - n - 1$.
\end{enumerate}
Then $M$ is matched to $N$.
\end{theorem}

\medskip

We now extend this result to general paving matroids by incorporating the hyperplane-nullity parameter. 
In particular, we generalize the sparse paving case ($t \leq 1$) to arbitrary paving matroids, 
with explicit dependence on the parameter~$t$.

\begin{theorem}\label{Asy pav}
Let $M$ be a matroid over $G$ and let $N$ be a paving matroid over $G$, both of the same rank $n$, with $0 \notin E(N)$. 
Let $t = \mathrm{null}(\mathcal{H}_N)$ with $0 < t \leq n$. 
Assume further that one of the following conditions holds:
\begin{enumerate}
    \item $|E(M)| < \min\{|E(N)| - 2t + 1,\, p(G)\}$, or
    \item $E(M)$ is not a progression, $G$ is finite, $|E(M)| = |E(N)| - 2t + 1$, and $|E(N)| < p(G)$, or
    \item $E(M)$ is neither a progression nor a semi-progression, $G$ is finite, $|E(M)| = |E(N)| - 2t + 2$, and $|E(N)| < p(G)$, or
    \item $|E(M)| < |E(N)| - n - 1$.
\end{enumerate}
Then $M$ is matched to $N$.
\end{theorem}

\begin{proof}[Proof of Theorem \ref{Asy pav} (1)-(3)]

We begin by addressing the following three special cases:

\begin{itemize}
    \item $M$ is a free matroid: In this case, $\mathcal{M} = E(M)$ is the unique basis of $M$. 
    Let $\mathcal{N}$ be a basis of $N$. 
    By Theorem~\ref{matchable sets, p(G)}, $\mathcal{M}$ is matched to $\mathcal{N}$ in the group-theoretic sense; 
    that is, there exists a bijection $f: \mathcal{M} \to \mathcal{N}$ such that $a + f(a) \notin \mathcal{M}$ for all $a \in \mathcal{M}$. 
    This implies that $\mathcal{M}$ is also matched to $\mathcal{N}$ in the matroid sense. 
    Hence, $M$ is matched to $N$.

    \item \textbf{$t = 1$:} By Lemma~\ref{sparse paving nullity}, $N$ is a sparse paving matroid. 
    Then, by Theorem~\ref{Asy sparse paving}, $M$ is matched to $N$.

    \item \textbf{$n = 1$:}  
    Let $\mathcal{M} = \{a\}$ be a basis of $M$. 
    Since $|E(M)| < |E(N)|$, there exists an element $b \in E(N)$ such that $a + b \notin E(M)$. 
    Define $\mathcal{N} = \{b\}$. 
    As $N$ is loopless, $\mathcal{N}$ is a basis of $N$. 
    Thus, $\mathcal{M}$ is matched to $\mathcal{N}$, and therefore $M$ is matched to $N$.
\end{itemize}

We now assume that $M$ is not a free matroid, that $t > 1$, and that $n > 1$. 
Suppose that one of the conditions (1)-(3) holds. 
Then $|E(M)| < p(G)$. 
Choose a subset $A \subseteq E(N)$ such that $|A| = |E(M)|$. 
Since $|A| = |E(M)| < p(G)$, Theorem~\ref{matchable sets, p(G)} guarantees the existence of a matching 
$f : E(M) \to A$ in the group-theoretic sense.  

Let $\mathcal{M} = \{a_1, \ldots, a_n\}$ be a basis of $M$, and define $b_i := f(a_i)$ for all $i\in [n]$. 
Set $\mathcal{N} = \{b_1, \ldots, b_n\} \subseteq E(N)$. 
Since $N$ is a paving matroid of rank $n$, the set $\mathcal{N}$ is either a basis or a circuit.  

If $\mathcal{N}$ is a basis, then $\mathcal{M}$ is matched to $\mathcal{N}$ in the matroid sense, and hence $M$ is matched to $N$, completing the proof in this case. 
If instead $\mathcal{N}$ is a circuit, we proceed by proving the following claim.

\begin{claim}\label{ant}
For every integer $s$ with $s\in [t]$, there exist distinct elements
$x_1,\dots,x_s\in E(N)\setminus \mathcal N$
and distinct indices $i_1,\dots,i_s\in [n]$ such that
\[
a_{i_j}+x_j\notin E(M)\qquad\text{for every }j\in [s].
\]
In particular, the conclusion holds for $s=t$.
\end{claim}

\begin{proof}[Proof of the Claim]
We proceed by induction on $s$.

\medskip
\noindent\textit{Base case ($s = 1$).} 
We show that there exist $x \in E(N) \setminus \mathcal{N}$ and $a \in \mathcal{M}$ such that $a + x \notin E(M)$. 
Suppose, for contradiction, that no such pair $(a, x)$ exists. 
Then
\[
\mathcal{M} + (E(N) \setminus \mathcal{N}) \subseteq E(M).
\]

Applying Theorem~\ref{Kneser} to the sets $\mathcal{M}$ and $E(N) \setminus \mathcal{N}$, we obtain a subgroup $H \leq G$ such that
\begin{align}
\label{eq10}
|\mathcal{M} + (E(N) \setminus \mathcal{N})| &\geq |\mathcal{M}| + |E(N) \setminus \mathcal{N}| - |H|, \\
\label{eq11}
\mathcal{M} + (E(N) \setminus \mathcal{N}) + H &= \mathcal{M} + (E(N) \setminus \mathcal{N}).
\end{align}

Since $\mathcal{M} + (E(N) \setminus \mathcal{N}) \subseteq E(M)$, it follows from \eqref{eq11} that
\[
|H| \leq |\mathcal{M} + (E(N) \setminus \mathcal{N}) + H| 
= |\mathcal{M} + (E(N) \setminus \mathcal{N})| 
\leq |E(M)| < p(G),
\]
so $H = \{0\}$. 
Substituting this into \eqref{eq10} yields
\begin{align}\label{eq14}
|E(M)| \;\geq\; |\mathcal{M}| + |E(N) \setminus \mathcal{N}| - 1 
= n + (|E(N)| - n) - 1 
= |E(N)| - 1 
> |E(N)| - 2t + 1.
\end{align}
We now consider three separate cases corresponding to conditions (1), (2), and (3) of the theorem.

\medskip
\noindent\textit{Case (1).} 
The condition 
\[
|E(M)| < \min\{|E(N)| - 2t + 1,\, p(G)\}
\] 
directly contradicts inequality~\eqref{eq14}.

\medskip

\noindent\textit{Case (2).} 
The assumption $|E(M)| = |E(N)| - 2t + 1$ directly contradicts inequality~\eqref{eq14}.

\medskip
\noindent\textit{Case (3).} 
Assume that $G$ is finite. 
Combining the assumption $|E(M)| = |E(N)| - 2t + 2$ with inequality~\eqref{eq14}, we obtain
\[
|E(M)| \;\geq\; |\mathcal{M} + (E(N) \setminus \mathcal{N})| 
   \;\geq\; |\mathcal{M}| + |E(N) \setminus \mathcal{N}| - 1 
   \;\geq\; |E(N)| - 2t + 2
   =|E(M)|.
\]
Therefore, 
\[
|\mathcal{M}| + |E(N) \setminus \mathcal{N}| - 1 \;=\; |\mathcal{M} + (E(N) \setminus \mathcal{N})| .
\]

The pair $(\mathcal{M},\, E(N) \setminus \mathcal{N})$ is critical and satisfies the conditions of Theorem~\ref{Kemperman's critical theorem} because:
\begin{itemize}
    \item Since $n>1$, we have $|\mathcal{M}| > 1$.
    \item $|E(N) \setminus \mathcal{N}| = |E(N)| - n > 1$.
    \item $|\mathcal{M}| + |E(N) \setminus \mathcal{N}| - 1 = |E(N)| - 1 \leq p(G) - 2$.
\end{itemize}

Hence, both sets are progressions with the same common difference. 
Since
\[
E(M) = \mathcal{M} + (E(N) \setminus \mathcal{N}),
\]
it follows that $E(M)$ is also a progression, a contradiction.

\noindent\textit{Inductive hypothesis.} 
Assume that distinct elements $x_1, \dots, x_{s-1} \in E(N) \setminus \mathcal{N}$ 
and distinct indices $i_1, \dots, i_{s-1} \in [n]$ have been chosen such that
\[
a_{i_j} + x_j \notin E(M) \quad \text{for every } j \in [s-1].
\]

\medskip
\noindent\textit{Inductive step.} 
We show that there exists
\[
x_s \in E(N) \setminus (\mathcal{N} \cup \{x_1, \dots, x_{s-1}\}) 
\quad \text{and} \quad 
i_s \in [n] \setminus \{i_1, \dots, i_{s-1}\}
\]
such that
\[
a_{i_s} + x_s \notin E(M).
\]

Suppose, for contradiction, that no such pair exists. 
Then
\begin{align}\label{eq12}
(E(N) \setminus (\mathcal{N} \cup \{x_1, \dots, x_{s-1}\})) 
+ (\mathcal{M} \setminus \{a_{i_1}, \dots, a_{i_{s-1}}\}) 
\subseteq E(M).
\end{align}

Applying Theorem~\ref{Kneser} to the sets 
$\mathcal{M} \setminus \{a_{i_1}, \dots, a_{i_{s-1}}\}$ 
and $E(N) \setminus (\mathcal{N} \cup \{x_1, \dots, x_{s-1}\})$, 
we obtain a subgroup $H \leq G$ such that
\begin{equation}\label{eq13}
\begin{aligned}
&\Big| (\mathcal{M} \setminus \{a_{i_1}, \ldots, a_{i_{s-1}}\}) 
+ (E(N) \setminus (\mathcal{N} \cup \{x_1, \ldots, x_{s-1}\})) \Big| \\
&\qquad \geq 
|\mathcal{M} \setminus \{a_{i_1}, \ldots, a_{i_{s-1}}\}| 
+ |E(N) \setminus (\mathcal{N} \cup \{x_1, \ldots, x_{s-1}\})| - |H|, \\[6pt]
&(\mathcal{M} \setminus \{a_{i_1}, \ldots, a_{i_{s-1}}\}) 
+ (E(N) \setminus (\mathcal{N} \cup \{x_1, \ldots, x_{s-1}\})) + H \\
&\qquad = 
(\mathcal{M} \setminus \{a_{i_1}, \ldots, a_{i_{s-1}}\}) 
+ (E(N) \setminus (\mathcal{N} \cup \{x_1, \ldots, x_{s-1}\})).
\end{aligned}
\end{equation}

In a similar manner as in the base case, we may argue that $H = \{0\}$. 
Combining this with inequalities~\eqref{eq12} and~\eqref{eq13}, we conclude that
\begin{equation} \label{eq16}
\begin{aligned}
|E(M)| &\geq 
\big| (\mathcal{M} \setminus \{a_{i_1}, \ldots, a_{i_{s-1}}\}) 
   + (E(N) \setminus (\mathcal{N} \cup \{x_1, \ldots, x_{s-1}\})) \big| \\
&\geq 
|\mathcal{M} \setminus \{a_{i_1}, \ldots, a_{i_{s-1}}\}| 
+ |E(N) \setminus (\mathcal{N} \cup \{x_1, \ldots, x_{s-1}\})| - 1 \\
&=|E(N)|- 2s + 1 \geq |E(N)| - 2t + 1.
\end{aligned}
\end{equation}

\medskip
We now consider three separate cases corresponding to conditions (1), (2), and (3) of the theorem.

\medskip
\noindent\textit{Case (1).} 
The condition 
\[
|E(M)| < \min\{|E(N)| - 2t + 1,\, p(G)\}
\] 
contradicts inequality~\eqref{eq16}.

\medskip
\noindent\textit{Case (2).} 
Assume $|E(M)| = |E(N)| - 2t + 1$. 
From~\eqref{eq16} we conclude that
\begin{multline*}
\Big| 
  (\mathcal{M} \setminus \{a_{i_1}, \ldots, a_{i_{s-1}}\}) 
  + \big(E(N) \setminus (\mathcal{N} \cup \{x_1, \ldots, x_{s-1}\})\big)
\Big| \\
= 
\big|\mathcal{M} \setminus \{a_{i_1}, \ldots, a_{i_{s-1}}\}\big|
+ \big|E(N) \setminus (\mathcal{N} \cup \{x_1, \ldots, x_{s-1}\})\big|
- 1.
\end{multline*}
Therefore 
\[
\big(\mathcal{M} \setminus \{a_{i_1}, \ldots, a_{i_{s-1}}\},\,
   E(N) \setminus (\mathcal{N} \cup \{x_1, \ldots, x_{s-1}\}) \big)
\]
is a critical pair satisfying the conditions of Theorem~\ref{Kemperman's critical theorem}, since:
\begin{itemize}
    \item $|\mathcal{M} \setminus \{a_{i_1}, \ldots, a_{i_{s-1}}\}| =n-(s-1)\geq n - (t-1) > 1$;
    \item As $M$ is not a free matroid and $t>1$, we have
\begin{align*}
|E(N) \setminus (\mathcal{N} \cup \{x_1, \ldots, x_{s-1}\})|
&= |E(N)|-(n+s-1) \\
&\ge |E(N)|-(n+t-1) \\
&= |E(M)|+2t-1-(n+t-1) \\
&> t>1.
\end{align*}
    \item 
   \[
\begin{aligned}
|\mathcal{M} \setminus \{a_{i_1}, \ldots, a_{i_{s-1}}\}|
&+ |E(N) \setminus (\mathcal{N} \cup \{x_1, \ldots, x_{s-1}\})| - 1 \\
&= |E(M)| \\
&\le |E(N)| - 1 \\
&\le p(G) - 2.
\end{aligned}
\]
\end{itemize}
By Theorem~\ref{Kemperman's critical theorem}, both sets are progressions with the same common difference. 
Thus
\[
E(M) = 
(\mathcal{M} \setminus \{a_{i_1}, \ldots, a_{i_{s-1}}\}) 
+ \big(E(N) \setminus (\mathcal{N} \cup \{x_1, \ldots, x_{s-1}\})\big)
\]
is itself a progression, contradicting the assumption that $E(M)$ is not a progression.

\medskip
\noindent\textit{Case (3).} 
Assume $|E(M)| = |E(N)| - 2t + 2$. 
From~\eqref{eq16} it follows that one of the following must hold:
\[
|\mathcal{M} \setminus \{a_{i_1}, \ldots, a_{i_{s-1}}\}| 
+ |E(N) \setminus (\mathcal{N} \cup \{x_1, \ldots, x_{s-1}\})| - 1 = |E(M)|,
\]
or
\[
|\mathcal{M} \setminus \{a_{i_1}, \ldots, a_{i_{s-1}}\}| 
+ |E(N) \setminus (\mathcal{N} \cup \{x_1, \ldots, x_{s-1}\})| - 1 = |E(M)| - 1.
\]
We now split into two subcases:

\begin{enumerate}
    \item If 
    \[
    \big|\mathcal{M} \setminus \{a_{i_1}, \ldots, a_{i_{s-1}}\}\big|
    + \big|E(N) \setminus (\mathcal{N} \cup \{x_1, \ldots, x_{s-1}\})\big| - 1 
    = |E(M)|,
    \]
    then, as in Case~(2), the pair 
    \[
    \big(\mathcal{M} \setminus \{a_{i_1}, \ldots, a_{i_{s-1}}\},\,
    E(N) \setminus (\mathcal{N} \cup \{x_1, \ldots, x_{s-1}\})\big)
    \]
    is critical and satisfies the conditions of Theorem~\ref{Kemperman's critical theorem}. 
    Hence both $\mathcal{M} \setminus \{a_{i_1}, \ldots, a_{i_{s-1}}\}$ and $E(N) \setminus (\mathcal{N} \cup \{x_1, \ldots, x_{s-1}\})$ are progressions with the same common difference. 
    Since
    \[
    E(M) = 
    (\mathcal{M} \setminus \{a_{i_1}, \ldots, a_{i_{s-1}}\}) 
    + (E(N) \setminus (\mathcal{N} \cup \{x_1, \ldots, x_{s-1}\})),
    \]
    it follows that $E(M)$ is itself a progression, a contradiction.

    \item If 
    \[
    \big|\mathcal{M} \setminus \{a_{i_1}, \ldots, a_{i_{s-1}}\}\big|
    + \big|E(N) \setminus (\mathcal{N} \cup \{x_1, \ldots, x_{s-1}\})\big| - 1 
    = |E(M)| - 1,
    \]
    then $|E(M)| = |E(N)|$. 
    On the other hand, we also know that $|E(M)| = |E(N)| - 2t + 2$, which forces $t = 1$, a contradiction.
\end{enumerate}

\bigskip

Thus the proof of Claim~\ref{ant} is complete. 
\end{proof}

Now let $x_1, \dots, x_t$ and $i_1, \dots, i_t$ be obtained as in Claim~\ref{ant}. 
Without loss of generality, we may assume that $i_j = j$ for every $j \in [t]$. 
That is, $a_j + x_j \notin E(M)$ for all $j \in [t]$. 
Set $X = \{x_1, \dots, x_t\}$ and $\mathcal{P} = X \cup \mathcal{N}$.

\begin{claim}
$r(\mathcal{P}) = n$.
\end{claim}

\begin{proof}[Proof of the Claim]
Assume for contradiction that $r(\mathcal{P}) \neq n$. 
Since $N$ is a paving matroid and $|\mathcal{P}| > n$, it follows that $r(\mathcal{P}) = n - 1$. 
Let $H = \mathrm{cl}(\mathcal{P})$. 
Then $H$ is a hyperplane, and we have
\[
t = \mathrm{null}(\mathcal{H}_N) 
   \geq \mathrm{null}(H) 
   = |H| - (n-1) 
   \geq |\mathcal{P}| - (n-1) 
   = (n+t) - (n-1) = t+1,
\]
a contradiction. 
Hence $r(\mathcal{P}) = n$.
\end{proof}

\medskip

We now define $\mathcal{N}_0 = \mathcal{N}$ and consider the following $n$-subsets of $E(N)$:
\[
\mathcal{N}_i = \{x_1, \dots, x_i, b_{i+1}, \dots, b_n\},
\quad \text{for all } i \in [t].
\]
\begin{claim}\label{existence of basis sym}
$\mathcal{N}_i$ is a basis of $N$ for some $i \in [t]$.
\end{claim} 
\begin{proof}[Proof of the Claim]

Suppose, for contradiction, that every $\mathcal{N}_i$ is not a basis. 
Since $N$ is a paving matroid, this means $r(\mathcal{N}_i) = n-1$. 
Set $H_i = \mathrm{cl}(\mathcal{N}_i)$ for $0 \leq i \leq t$. 
Clearly, each $H_i$ is a hyperplane of $N$. 
On the other hand,
\[
|H_i \cap H_{i+1}| \;\geq\; |\mathcal{N}_i \cap \mathcal{N}_{i+1}| = n-1.
\]
Invoking Theorem~\ref{(n-1) partitions}, we conclude that $H_i = H_{i+1}$ for all $0 \leq i \leq t-1$. 
Thus,
\[
\mathcal{P} \subseteq \bigcup_{i=0}^{t} H_i = H_0,
\]
implying $r(\mathcal{P}) \leq r(H_0) = n-1$. 
This contradicts the fact that $r(\mathcal{P}) = n$. 
Hence, $\mathcal{N}_i$ is a basis of $N$ for some $i \in [t]$.

\end{proof}
\medskip

Assume that $i \in [t]$ is obtained as in Claim~\ref{existence of basis sym}. 
Then $\mathcal{M}$ is matched to $\mathcal{N}_i$ via the assignment
\[
a_j \mapsto x_j \quad \text{for } 1 \leq j \leq i,
\qquad
a_k \mapsto b_k \quad \text{for } i+1 \leq k \leq n.
\]
Therefore $M$ is matched to $N$.
\end{proof}

\begin{remark}
    It is worth noting that, although the conditions in Theorem \ref{Asy pav}-(3), such as $E(M)$ not being a semi-progression may seem unused in the proof of Theorem \ref{Asy pav}, they are in fact applied in the special case $t=1$. In this instance, we invoke Theorem \ref{Asy sparse paving}, where this condition plays a role.

\end{remark}

Next, we complete the proof of Theorem \ref{Asy pav}-(4). In this case, the condition 
\(|E(N)| < p(G)\) is not imposed. Hence, unlike parts (1)--(3), where Theorem \ref{matchable sets, p(G)} 
is applicable, we cannot invoke it here. To establish the existence of matchings in the group sense 
(and subsequently in the matroid sense), we instead employ elementary techniques from additive number 
theory, followed by standard matroidal arguments.

\begin{proof}[Proof of Theorem \ref{Asy pav} (4)]
Let $\mathcal{M}=\{a_1,\cdots, a_n\}$ be a basis for $M$. We show that $\mathcal{M}$ is matched to a basis of $N$.

\begin{claim}
There exist $n$ distinct elements $b_1,b_2,\dots,b_n \in E(N)$ such that  
\[
a_i+b_i \notin E(M) \quad \text{for all } i \in [n].
\]
\end{claim}

\begin{proof}[Proof of the Claim]
We first choose $b_1 \in E(N)$ so that $a_1+b_1 \notin E(M)$. Such $b_1$ exists since $|E(M)| < |E(N)|$. Next, choose $b_2 \in E(N)\setminus\{b_1\}$ such that $a_2+b_2 \notin E(M)$. Such $b_2$ exists as $|E(M)| < |E(N)|-1$. 

We proceed by induction on $n$. Suppose $b_1,\dots,b_j$ have been chosen. We show that there exists $b_{j+1} \in E(N)\setminus\{b_1,\dots,b_j\}$ satisfying $a_{j+1}+b_{j+1} \notin E(M)$. Assume to the contrary that no such $b_{j+1}$ exists. Then
\[
a_{j+1}+\big(E(N)\setminus\{b_1,\dots,b_j\}\big) \subseteq E(M),
\]
which implies 
\[
|E(N)\setminus\{b_1,\dots,b_j\}| \leq |E(M)|.
\]
Thus 
\[
|E(M)| \geq |E(N)|-j \geq |E(N)|-n,
\]
a contradiction. This proves the claim.
\end{proof}

Let $\mathcal{N}=\{b_1,\dots,b_n\}$, where $b_1,\dots,b_n$ are as above. Since $N$ is a paving matroid, then $\mathcal{N}$ is either a basis of $N$ or a circuit of rank $n-1$. If $\mathcal{N}$ is a basis, then $\mathcal{M}$ is matched to $\mathcal{N}$ and we are done. Suppose $\mathcal{N}$ is a circuit of rank $n-1$.

\begin{claim}\label{cl:extra-element}
There exist $x \in E(N)\setminus \mathcal{N}$ and $i \in [n]$ such that $a_i+x \notin E(M)$.
\end{claim}

\begin{proof}[Proof of Claim]
Assume to the contrary that no such $x$ and $a_i$ exist. Then
\[
M+\big(E(N)\setminus\{b_1,\dots,b_n\}\big) \subseteq E(M),
\]
which implies
\[
|E(N)|-n = \big|E(N)\setminus\{b_1,\dots,b_n\}\big| \leq |E(M)|.
\]
This contradicts the assumption $|E(M)| < |E(N)|-n-1$.
\end{proof}

Choose $i \in [n]$ and $x \in E(N)\setminus \mathcal{N}$ as in Claim~\ref{cl:extra-element}. Consider the subset 
\[
\mathcal{P}=\{b_1,\dots,b_{i-1},x,b_{i+1},\dots,b_n\} \subseteq E(N).
\]
In a similar manner as in the proof of Theorem~\ref{Asy pav} (parts (1)-(3)), one may argue that $\mathcal{P}$ is a basis of $N$. Evidently, $\mathcal{M}$ is matched to $\mathcal{P}$. Hence $M$ is matched to $N$, completing the proof.

\end{proof}
\begin{remark}
   Note that Theorem~\ref{Asy pav} does not address the case $t=0$, since this forces $N$ to be a uniform matroid (see Corollary~\ref{small nullity}). Matchability in uniform matroids is straightforward and has been completely characterized in \cite[Proposition~2.9]{Zerbib0}.
 
\end{remark}

\newcommand{\wh}[1]{\widehat{#1}}

\newcommand{\ol}[1]{\overline{#1}}
\begin{remark}
Note that in the asymmetric problem of the sparse paving case, dependence is controlled by circuit-hyperplanes, whereas in the general paving case, the size of the obstruction is measured by hyperplane nullity. Theorem~\ref{Asy pav} shows that the existence of matchings is governed not merely by the size gap \( |E(N)| - |E(M)| \), but also by the extent to which \(N\) deviates from uniformity, as quantified by the hyperplane nullity parameter \(t\).
\end{remark}
\begin{example}
We define two matroids $M$ and $N$ over the cyclic group $\mathbb{Z}_{13}$.

Let $E(M)=\{\ol{2},\ol{4},\ol{6},\ol{11}\}$ and
define matroid $M$ of rank 3 by its set of bases
\[
\mathcal{B}(M)=\binom{E(M)}{3}\setminus\big\{\{\ol{2},\ol{4},\ol{6}\}\big\}.
\]
Equivalently, the unique $3$-circuit is $\{\ol{2},\ol{4},\ol{6}\}$, and the bases are
$\{\ol{2},\ol{4},\ol{11}\}$, $\{\ol{2},\ol{6},\ol{11}\}$, and $\{\ol{4},\ol{6},\ol{11}\}$.

Let $E(N)=\{\ol{1},\ldots,\ol{9}\}$ and fix $H=\{\ol{1},\ol{2},\ol{3},\ol{4}\}\subseteq E(N)$.
Define a rank-$3$ matroid $N$ by
\[
\mathcal{B}(N)=\binom{E(N)}{3}\setminus \binom{H}{3},
\]
i.e., the only $3$-subsets that are not bases are the four triples inside $H$.
Then $N$ is a paving matroid that is not sparse paving. Its largest hyperplane is $H$ (size $4$), so its
hyperplane nullity is $|H|-r(H)=4-2=2$.
Thus $M$ and $N$ satisfy Condition (1) of Theorem~\ref{Asy pav} and $M$ must be matched to $N$.

\medskip
\noindent\textit{Exhibiting matching bases.}
We now explicitly show that $M$ is matched to $N$ by pairing each basis of $M$ with a basis of $N$:
\begin{itemize}
    \item $\{\ol{2},\ol{4},\ol{11}\}$ in $M$ matches $\{\ol{1},\ol{3},\ol{7}\}$ in $N$;
    \item $\{\ol{2},\ol{6},\ol{11}\}$ in $M$ matches $\{\ol{1},\ol{4},\ol{9}\}$ in $N$;
    \item $\{\ol{4},\ol{6},\ol{11}\}$ in $M$ matches $\{\ol{1},\ol{6},\ol{9}\}$ in $N$.
\end{itemize}
\end{example}

\subsection{Matchable Bases through Stressed Hyperplanes}\label{stressed} 

Let $E$ be a nonempty set. Recall that A matroid basis system is a nonempty family $\mathcal{B}\subseteq 2^{E}$ of subsets, called bases, that satisfies the exchange axiom: for all distinct $B,B'\in\mathcal{B}$ and every $x\in B\setminus B'$, there exists $y\in B'\setminus B$ such that $(B\setminus\{x\})\cup\{y\}\in\mathcal{B}$. It follows that $\mathcal{B}$ is the set of bases of a matroid.

Following \cite{Ferroni}, a hyperplane $H$ of a matroid $M$ is said to be \emph{stressed} if every $(r(M)-1)$-subset of $H$ is independent.

\begin{proposition}[{\cite{Ferroni}}]\label{Stressed Hyperplane}
A matroid is paving if and only if all its hyperplanes are stressed.
\end{proposition}

\begin{definition}
Let $M$ be a matroid of rank $n=r(M)$ with bases $\mathcal{B}$, and let $S\subseteq E(M)$. 
Denote by $\binom{S}{n}$ the family of $n$-subsets of $S$. 
Suppose that no element of $\binom{S}{n}$ is a basis of $M$. 
We say that $S$ \emph{can be relaxed} if
\[
\mathrm{Rel}_S(\mathcal{B}) \;:=\; \mathcal{B}\,\cup\,\binom{S}{n}
\]
is again a matroid basis system. In this case, the resulting matroid is called the \emph{relaxation of $M$ at $S$}, and is denoted by $\mathrm{Rel}_S(M)$.
\end{definition}

\begin{proposition}[{\cite{Ferroni}}]\label{stressed & relaxed}
If $H$ is a stressed hyperplane of $M$, then $H$ can be relaxed.
\end{proposition}

As a consequence, given an $n$-rank paving matroid $M$ on $m$ elements, one can obtain the uniform matroid $U_{n,m}$ by repeatedly relaxing its stressed hyperplanes. Indeed, \cite{Ferroni} shows that every relaxation of a paving matroid is again paving, and that relaxing all stressed hyperplanes of $M$ yields $U_{n,m}$.

\medskip

\begin{theorem}\label{Asy paving}
Let $M$ be a matroid and $N$ a paving matroid over an abelian group $G$, both of rank $n$, with $0\notin E(N)$ and
\[
|E(M)| \;\leq\; \min\{\,p(G)-1,\,|E(N)|\,\}.
\]
If $\mathcal{M}$ is a basis of $M$, then $\mathcal{M}$ is either matched to a basis of $N$, or to a basis of $\mathrm{Rel}_S(N)$ for some $S\subseteq E(N)$.
\end{theorem}

\begin{proof}
Let $A\subseteq E(N)$ with $|A|=|E(M)|$. By Theorem~\ref{matchable sets, p(G)}, $E(M)$ is matched to $A$ in the group sense; equivalently, there exists a matching $f:E(M)\to A$.
Let $\mathcal{M}=\{a_1,\dots,a_n\}$ and set $b_i=f(a_i)$ for $i\in[n]$, and $\mathcal{N}=\{b_1,\dots,b_n\}$.

If $\mathcal{N}$ is a basis of $N$, then $\mathcal{M}$ is matched to $\mathcal{N}$ in the matroid sense and we are done. 
Otherwise, set $H=\operatorname{cl}(\mathcal{N})$. Since $N$ is paving, $H$ is a hyperplane of $N$, and by Proposition~\ref{Stressed Hyperplane} it is stressed. Hence, by Proposition~\ref{stressed & relaxed}, $H$ can be relaxed. By definition of relaxation, $\mathcal{N}$ is a basis of $\mathrm{Rel}_H(N)$ (equivalently, $\mathcal{N}\in \mathrm{Rel}_H(\mathcal{B}_N)$, where $\mathcal{B}_N$ is the basis system of $N$). Consequently, $\mathcal{M}$ is matched to $\mathcal{N}$ in the matroid sense in $\mathrm{Rel}_H(N)$, as required.
\end{proof}
\begin{remark}
    When matchings between matroid bases are not available, stressed hyperplanes
show that compatibility can often be recovered by relaxing independence constraints.
A stressed hyperplane may be viewed as a minimal obstruction to uniformity. Relaxing
such hyperplanes, turning certain dependent sets into bases, gradually transforms a
paving matroid toward a uniform matroid while preserving structure. Theorem~\ref{Asy paving}
shows that this process offers an alternative route to matchability: even if bases of
$M$ cannot be matched to those of $N$ directly, they may be matched to bases of a
relaxation $\mathrm{Rel}_S(N)$. Thus matchability is not a fragile property tied to exact
independence, but a stable phenomenon that persists under natural modifications.
Relaxation therefore links the restrictive setting of paving matroids with the uniform
case, where matching problems are straightforward.

\end{remark}

\subsection{Concluding Remarks and Future Directions}
\begin{enumerate}
    \item Let $M$ be a matroid of rank $n$ and let $k\in[n]$. 
We say that $M$ is \emph{$k$-paving} if every circuit of $M$ has size $> n-k$. For $k=1$ this recovers the usual paving condition, 
where all circuits have size at least $n$. Clearly, the following hierarchy holds:
\[
\text{paving} = 1\text{-paving} \ \subseteq\ 2\text{-paving} \ \subseteq\ \cdots\ \subseteq\ 
n\text{-paving} = \{\text{all matroids of rank }n\}.
\]

The notion of $k$-paving matroids was first introduced in \cite{Rajpal}, where a characterization 
was given; see also \cite{Singh} for more recent developments. We believe that some of the techniques 
developed here for the matchability of paving matroids could be adapted to investigate 
matchability in the broader class of $k$-paving matroids for $k>1$. 

In particular, one may ask whether the dependence on $t$ in Theorem \ref{Asy pav} extends in a clean way:
\[
|E(M)| < |E(N)| - 2t + 1 
\;\;\longrightarrow\;\; 
|E(M)| < |E(N)| - c(k,t) + 1,
\]
for some explicit function $c(k,t)$. Further work in this direction 
appears promising.
\item For a paving matroid $N$, one may obtain the uniform matroid $U_{n,m}$ by 
successively relaxing its stressed hyperplanes, forming a hierarchy of relaxations
\[
N \;\longrightarrow\; \mathrm{Rel}_{H_1}(N) \;\longrightarrow\; 
\; \cdots \;\longrightarrow\; U_{n,m}.
\]
Assuming the conditions of Theorem \ref{Asy paving}, what is the minimal number of steps in this hierarchy needed to guarantee that $M$ can be matched to some consecutive relaxations of $N$? Clearly, this process must terminate, since in the worst-case scenario \(M\) is eventually matched to \(U_{n,m}\).
Is a single relaxation always sufficient, or can certain pairs $(M,N)$ require 
multiple relaxations within the hierarchy?
\end{enumerate}

\bigskip

Data sharing. Data sharing is not applicable to this article as no datasets were generated or analyzed.\par
Conflict of interest. To the best of our knowledge, no conflict of interest, whether of a financial or personal nature, has influenced the work presented in this article.

\end{document}